\documentclass[english,13pt]{article}
\usepackage[T1]{fontenc}
\usepackage[latin1]{inputenc}
\usepackage{geometry}
\usepackage{float}
\usepackage{mathrsfs}
\usepackage{amsmath}
\usepackage{amssymb}
\usepackage{graphicx}
\usepackage{bbm}
\usepackage{hyperref}
\hypersetup{
    colorlinks=true,
    linkcolor=blue,
    filecolor=magenta,      
    urlcolor=cyan,
    citecolor = red,
}

\makeatletter

\usepackage{algorithm,algpseudocode}

\usepackage{amsthm}

\usepackage{mathrsfs}

\usepackage{amsfonts}

\usepackage{epsfig}

\usepackage{bm}

\usepackage{mathrsfs}

\usepackage{enumerate}

\numberwithin{equation}{section}

\@ifundefined{definecolor}{\@ifundefined{definecolor}
 {\@ifundefined{definecolor}
 {\usepackage{color}}{}
}{}
}{}

\usepackage{subfig}\usepackage[all]{xy}

\newtheorem{theorem}{Theorem}[section]
\newtheorem{lem}{Lemma}[section]
\newtheorem{rem}{Remark}[section]
\newtheorem{prop}{Proposition}[section]
\newtheorem{cor}{Corollary}[section]
\newcounter{hypA}
\newenvironment{hypA}{\refstepcounter{hypA}\begin{itemize}
  \item[({\bf A\arabic{hypA}})]}{\end{itemize}}
\newcounter{hypB}

\newcounter{hypH}
\newenvironment{hypH}{\refstepcounter{hypH}\begin{itemize}
 \item[({\bf H\arabic{hypH}})]}{\end{itemize}}

\usepackage{babel}\date{}

\usepackage{babel}

\makeatother

\usepackage{babel}

\newcommand{\EE}{\mathbb{E}}

\newcommand{\be}{\begin{equation}}
\newcommand{\ee}{\end{equation}}

\begin{document}

\begin{center}

{\Large \textbf{On the Particle Approximation of Lagged Feynman-Kac Formulae}}

\vspace{0.5cm}

	BY  ELSIDDIG AWADELKARIM$^{1}$,  MICHEL CAFFAREL$^{2}$,  PIERRE DEL MORAL$^{3}$ \& AJAY JASRA$^{4}$

{\footnotesize $^{1}$Applied Mathematics and Computational Science Program,  Computer, Electrical and Mathematical Sciences and Engineering Division, King Abdullah University of Science and Technology, Thuwal, 23955-6900, KSA.}\\
	{\footnotesize $^{2}$University of Toulouse,  CNRS-Lab. de Chimie et Physique Quantiques, Toulouse,  31062, FR.}\\
{\footnotesize $^{3}$Centre de Recherche Inria Bordeaux Sud-Ouest, Talence, 33405, FR.}\\
{\footnotesize $^{4}$School of Data Science,  The Chinese University Hong Kong, Shenzhen, CN.}\\
{\footnotesize E-Mail:\,} \texttt{\emph{\footnotesize elsiddigawadelkarim.elsiddig@kaust.edu.sa;
caffarel@irsamc.ups-tlse.fr; \\ pierre.del-moral@inria.fr;
ajayjasra@cuhk.edu.cn }}

\end{center}

\begin{abstract}
	In this paper we examine the numerical approximation of the limiting invariant measure associated with Feynman-Kac formulae.  
These are expressed in a discrete time formulation and are associated with a Markov chain and a potential function.
The typical application considered here is the computation of eigenvalues associated with non-negative operators 
as found, for example, in physics or particle simulation of rare-events. We focus on a novel \emph{lagged} approximation of this invariant measure, 
based upon the introduction of a ratio of time-averaged Feynman-Kac marginals associated with 
a positive operator iterated $l \in\mathbb{N}$ times; a lagged Feynman-Kac formula. This estimator and its approximation using Diffusion Monte Carlo (DMC) 
	have been extensively employed in the physics literature.  
In short, DMC is an iterative algorithm involving $N\in\mathbb{N}$ particles or walkers 
simulated in parallel, that undergo sampling and resampling operations. 
In this work, 
it is shown that 
for the DMC approximation of the lagged Feynman-Kac formula, 
one has an almost sure characterization of the $\mathbb{L}_1$-error as the time parameter (iteration) goes to
infinity 
and this
is at most of $\mathcal{O}(\exp\{-\kappa l\}/N)$, for $\kappa>0$.  
In addition a non-asymptotic in time,  and time uniform $\mathbb{L}_1-$bound is proved 
which is $\mathcal{O}(l/\sqrt{N})$.
We also prove a novel central limit theorem to give a characterization of the exact
asymptotic in time variance.
	This analysis demonstrates that the strategy used in physics,  namely, to run DMC with $N$ and $l$ small and, for long time enough,   is mathematically justified.  
 Our results also suggest how one should choose $N$ and $l$ in practice.
We emphasize that these results are not restricted to physical applications; they have broad relevance to the general problem 
of particle simulation of the Feynman-Kac formula, which is utilized in a great variety of scientific and engineering fields.
\end{abstract}

\section{Introduction}
The Feynman-Kac (FK) formula plays a central role in many scientific fields including mathematics \cite{beskos1},  statistics \cite{ddj06}, physics \cite{caffarel_1988,Simon},
control theory \cite{delmoral_2000},  quantitative finance \cite{huyen_2009}, to name some of the main ones, see for instance \cite{delm,delm:13} for 
book-length introductions to the subject. In the context of this work, the Feynman-Kac formula can be formulated in the following way. 
Consider a time-homogeneous Markov
chain with initial distribution $\eta_0$ and transition kernel $M$ on a measurable state-space $\mathsf{E}$ 
and let $G$ be a bounded, measurable and strictly positive real-valued function on $\mathsf{E}$.
The Feynman-Kac expresses the evolution of the initial distribution as follows. Let $n \in\mathbb{N}$ and any $\varphi:\mathsf{E}\rightarrow\mathbb{R}$ bounded and measurable, the Feynman-Kac measure $\eta_n(\varphi)$ is given by
\be
\eta_n(\varphi) := \frac{\mathbb{E}\left[\varphi(X_n)\prod_{p=0}^{n-1} G(X_p)\right]}{\mathbb{E}\left[\prod_{p=0}^{n-1} G(X_p)\right]}
\label{FK}
\ee
where $\mathbb{E}[\cdot]$ is the expectation operator w.r.t. the probability law of the afore-mentioned Markov chain. In this work, we will be 
in particular interested in evaluating $\eta_{\infty}(\varphi):=\lim_{n\rightarrow+\infty}\eta_n(\varphi)$ which exists under certain 
mathematical assumptions. However, except for trivial cases, the quantity $\eta_n(\varphi)$ cannot be evaluated analytically 
and one must rely upon numerical methods. These approaches are referred to as sequential Monte Carlo or particle filters 
in mathematics/statistics and under the general denomination of 
Diffusion Monte Carlo (DMC) in physics. Several variants of DMC have been introduced such as
Green's function Monte Carlo \cite{binder, kalos1962monte},  Fixed-Node Diffusion Monte Carlo \cite{reynolds1982}, Pure Diffusion Monte Carlo
\cite{caffarel_1988, caffarel1986treatment}, 
Stochastic Reconfiguration Monte Carlo \cite{assaraf2000diffusion,heth,sorr,sorella2000green} and Reptation Monte Carlo \cite{moroni1999}, to cite 
the main ones.

In practice, calculations are done by introducing a McKean interacting particle system involving a number of particles to which stochastic rules 
built from $M$ (sampling using the Markov kernel)  and $G$ (resampling using the weight $G$) are applied iteratively. 
In the mathematical literature, 
these methods are by now rather well understood, with a plethora of results; see for instance
\cite{beskos,delm,delm:13,dhj,whiteley1}.
In the following where the estimator discussed is borrowed to physics, 
we will adopt the name DMC for the preceding numerical algorithm. Note that 
in physics, the $N$ samples/particles of the evolving population are called walkers, so the two terms "particles" and 
"walkers" will be used interchangeably in this article.
Now, a key point is that a great freedom exists in choosing either the 
the stochastic rules and/or the estimator employed for $\eta_n(\varphi)$. The problem of making some optimal choice leading to smaller variances 
and reduced finite-population biases is thus a central issue in numerical implementations.

In this work we analyze an original estimator for the Feynman-Kac measure introduced in physics \cite{caffarel_1988, caffarel1986treatment} and which appears, to the best of our knowledge, 
not to have been considered in the mathematical literature. In physics, the problem considered consists in solving the imaginary-time dependent 
Schr\"odinger equation using the Feynman-Kac formula.

In its simplest formulation,
the Feynman-Kac formula is a particular case of Eq.(\ref{FK}) with the following setting:
i) $Q(x,y)=G(x)M(x,dy)$ is the kernel of the operator $e^{-\tau H}$ where $H$ is the Hamiltonian operator describing the system and $\tau$ a small positive 
quantity playing the role of a time step, ii) $M(x,dy)$ the Markov kernel corresponding to the brownian process,
and iii) $G(x)=e^{-\tau V({\bf x})}$ where $V({\bf x})$ is the potential function. 
In practice, for non-trivial quantum systems,  calculations are possible (small enough statistical errors) 
only if importance sampling is introduced. We are then led to a so-called importance sampled 
Feynman-Kac formula, which is still a particular case of Eq.(\ref{FK}), but with some more general operator $Q$. 
Details are given in Appendix \ref{FKP}.
A general presentation of this setting and the way it is derived in physics can be found for example in \cite{caffarel_1988,Simon}. 
As known, the quantity 
$\eta_{\infty}(G)$ obtained from Eq.(\ref{FK}) turns out to be the eigenvalue of the operator $Q$, that is the $\lambda\in\mathbb{R}^+$ such that for some $h:\mathsf{E}\rightarrow\mathbb{R}^+$, $\int_{\mathsf{E}}h(y) Q(x,dy) = \lambda h(x)$; the computation of which has found several 
applications in mathematics, for example, in rare-events estimation \cite{dd04,whiteley}. 
In physics, getting $\lambda$ gives the ground-state energy, $E_0$,
via the relation $\lambda= e^{-\tau E_0}$ and $h$ is
the ground-state eigenstate from which physical properties can be computed \cite{rousset}.


Based upon the Feynman-Kac formula representation, different FK-like formulae can be proposed to evaluate $\eta_{\infty}(\varphi)$.
For instance, one can consider the following quantity
\begin{equation}\label{eq:standard}
\overline{\eta}_n(\varphi):=\frac{1}{n}\sum_{k=0}^{n-1} \eta_k(\varphi),
\end{equation}
which, as $n$ grows large, will converge, under assumptions, to $\eta_{\infty}(\varphi)$.  In this work, we introduce the 
alternative quantity built from $\eta_k$ and defined as follows.
Let $l\in\mathbb{N}$
be a lag, which can be chosen to be any fixed value and define the function for $x_0\in\mathsf{E}$:
$$
Q^l(\varphi)(x_0) := \int_{\mathsf{E}^l} \varphi(x_l) \prod_{k=1}^l Q(x_{k-1},dx_k).
$$
We define the quantity
\begin{equation}\label{eq:lagged}
\frac{\overline{\eta}_n(Q^l(\varphi))}{\overline{\eta}_n(Q^l(1))}
\end{equation}
which leads to the following FK-like formula for the limiting invariant measure
\be
\eta_{\infty}(\varphi) = \lim_{n\rightarrow \infty} \frac{\overline{\eta}_n(Q^l(\varphi))}{\overline{\eta}_n(Q^l(1))}.
\ee
In the following, we will refer to \eqref{eq:lagged} as the \emph{lagged} Feynman-Kac formula.

The general idea of why one may prefer to apply the numerical DMC algorithm to \eqref{eq:lagged} instead of \eqref{eq:standard} 
is the flexibility resulting from the introduction of the lag parameter $l$. For example, and without entering into details at this stage, 
using DMC for approximating \eqref{eq:standard} necessitates 
in general a large number $N$ of walkers to extrapolate to zero the population bias of order $\frac{1}{N}$
as $N$ goes to infinity. In sharp contrast, this problem becomes much less severe using \eqref{eq:lagged} since, as we shall show,
the $N$-dependence of the results decreases rapidly as the lag $l$ increases. 
From that, one might obtain a substantial computational saving in the DMC simulation.
Finally, as already mentioned and to the best of our knowledge, 
a mathematical study of the DMC approximation of \eqref{eq:lagged} has not been undertaken; this is the topic of the present study.

\subsection{Summary of Main Results and Article Structure}

The analysis of the DMC approximation of \eqref{eq:lagged} is rather challenging.  As we will show later on,  this consists of a normalized estimator of walkers over the lag $l$,  time-averaged in as in \eqref{eq:lagged}. The standard theory associated to DMC (e.g.~\cite{delm}) is not trivially applicable to the estimator that is used in practice,  hence requiring an original approach.  In this article, under assumptions,  we prove the following interesting results:
\begin{itemize}
\item{For a fixed number of particles $N$ we prove an almost sure (as $n$ grows) limit for the $\mathbb{L}_1-$error of the  DMC approximation of \eqref{eq:lagged}  and show that this is at most $\mathcal{O}(\exp\{-\kappa l\}/N)$, $\kappa>0$.  This is Theorem \ref{prop:main_prop} and Corollary \ref{cor:main_cor}.}
\item{For a fixed number of particles $N$ and time horizon $n$
we prove that the appropriately centered $\mathbb{L}_1-$error
DMC approximation of \eqref{eq:lagged}  is at most $\mathcal{O}(l/\sqrt{N})$. 
The bound is uniform in time ($n$).
This is Theorem \ref{theo:l1}.}
\item{For a fixed number of particles $N$ we prove an asymptotic in $n$ central limit theorem (CLT) for the estimator.  This is in Theorem \ref{theo:clt1}.}
\end{itemize}
The implication of these results are rather important.  Theorem \ref{prop:main_prop} and 
Corollary \ref{cor:main_cor}
establish that,  as is done in practice,  DMC can be run with $N$ the number of walkers small and $l$ small,  for long-times and the error of the approximation can be vanishingly (exponentially) small.  
Theorem \ref{theo:l1} also shows that one should not necessarily grow $l$ arbitrarily large as there is a price to pay.  We also use our theorems to compare with the DMC
approximation of \eqref{eq:standard}.  
Theorem \ref{theo:clt1} provides an exact characterization of the asymptotic in time variance.  Our results are also confirmed in numerical simulations.
The proofs of the afore mentioned results require some non-standard arguments based upon novel path-wise solutions of the Poisson equation for Markov chains.  Our proofs are not confined to problems in physics and indeed,  apply to any application where the particle approximation
of \eqref{eq:lagged} is of practical interest.  This constitutes many applications in rare-events estimation and
stochastic control problems; see \cite{dd04,whiteley} for instance.

This paper is structured as follows.  In Section \ref{sec:fkform} we detail the formulae of interest and the Diffusion Monte Carlo algorithm.  In Section \ref{sec:est} we give the numerical estimator under study.  Section \ref{sec:main_res} presents our mathematical results with discussion.  The appendix contains the technical details associated to the proofs of our main results as well as some details of the application in physics.

\section{Feynman-Kac Formulae and Diffusion Monte Carlo}\label{sec:fkform}

\subsection{Notation}

Let $(\mathsf{E},\mathcal{E})$ be a measurable space. Denote by $\mathcal{P}(\mathsf{E})$
and $\mathcal{M}(\mathsf{E})$ the collection of probability measures and non-negative measures on $(\mathsf{E},\mathcal{E})$.  For $\varphi:\mathsf{E}\rightarrow\mathbb{R}$ we write $\mathcal{B}_b(\mathsf{E})$ as the collection of bounded and measurable functions on $(\mathsf{E},\mathcal{E})$. For $\varphi\in\mathcal{B}_b(\mathsf{E})$, the sup-norm is written $\|\varphi\|\sup_{x\in\mathsf{E}}|\varphi(x)|$.
For $(\mu,\varphi)\in\mathcal{M}(\mathsf{E})\times\mathcal{B}_b(\mathsf{E})$, $\mu(\varphi)=\int_{\mathsf{E}}\varphi(x)\mu(dx)$.
Let $P:\mathsf{E}\rightarrow\mathcal{M}(\mathsf{E})$, $(\mu,\varphi)\in\mathcal{M}(\mathsf{E})\times\mathcal{B}_b(\mathsf{E})$ we use the short-hand notation for $x\in\mathsf{E}$:
$$
\mu P(\varphi) = \int_{\mathsf{E}}\mu(dx) ~P(\varphi)(x) \quad\mbox{\rm and}\quad P(\varphi)(x) = \int_{\mathsf{E}} P(x,dy) \varphi(y).
$$
Note that one can use the notation $\mu P(\varphi)$ and $\mu(P(\varphi))$ exchangeably,  but we generally use the former.
$\mathbb{N}_0=\mathbb{N}\cup\{0\}$. For $P:\mathsf{E}\rightarrow\mathcal{M}(\mathsf{E})$, $\varphi\in\mathcal{B}_b(\mathsf{E})$ and $n\in\mathbb{N}$, $P^n(\varphi)(x_0)=\int_{\mathsf{E}^n}
\varphi(x_n) \prod_{k=1}^n P(x_{k-1},dx_k)$ with the convention that 
$P^0(\varphi)(x_0)=\varphi(x_0)$. For $A\in\mathcal{E}$ the Dirac measure is written $\delta_A(dx)$, with the convention that if $A=\{x\}$, $x\in\mathsf{E}$, we write $\delta_x(dy)$.
For $(\mu,\nu)\in\mathcal{P}(\mathsf{E})^2$,  we denote by $\|\mu-\nu\|_{\tiny \textrm{tv}}=\sup_{A\in\mathcal{E}}|\mu(A)-\nu(A)|$ as the total variation distance.
$\mathcal{N}(0,\sigma^2)$ denotes the one-dimensional normal distribution of mean 0 and variance $\sigma^2>0$. 

\subsection{Feynman-Kac Semigroups}
\label{FKSG}

Consider the  Feynman-Kac measures on a measurable space $(\mathsf{E},\mathcal{E})$, for $(n,\varphi)\in\mathbb{N}_0\times\mathcal{B}_b(\mathsf{E})$:
$$
\eta_n(\varphi)=\gamma_n(\varphi)/\gamma_n(1)
\quad\mbox{\rm with}\quad
\gamma_n(\varphi):=\EE\left(\varphi(X_n)\prod\limits_{p=0}^{n-1}G(X_p) \right)
$$
where $G:\mathsf{E}\rightarrow\mathbb{R}^+$ is a strictly positive and bounded potential and the expectation $\mathbb{E}(\cdot)$
is w.r.t.~the law of a time-homogenous Markov chain with initial distribution $\eta_0$
and transition kernel $M:\mathsf{E}\rightarrow\mathcal{P}(\mathsf{E})$.
We use the convention $\gamma_0=\eta_0=\mbox{\rm Law}(X_0)$, for $n=0$.
The unnormalized measures $\gamma_n$ have a linear evolution given for any $n\geq 1$ by the recursion 
\begin{equation}\label{def-gamma}
\gamma_n(\varphi)=\gamma_{n-1}Q(\varphi)
\quad \mbox{\rm and we have}\quad
\gamma_n(\varphi)=\eta_n(\varphi)~\prod_{0\leq p<n}\eta_p(G)
\end{equation}
where $Q(x,dy)=G(x)M(x,dy)$.
The normalized measures have a nonlinear evolution
\begin{equation}\label{def-eta}
\eta_{n+1}(\varphi)=\Phi(\eta_{n})(\varphi)=\Psi_{G}(\eta_{n})M(\varphi)\quad \mbox{\rm with}\quad
\Psi_{G}(\eta_n)(\varphi):=\frac{\eta_n(G\varphi)}{\eta_n(G)}.
\end{equation}
The evolution semigroup associated with the flow $\eta_n$ is given for any $p\leq n$  by
$$
\eta_{n}(\varphi)=\Phi^{n-p}(\eta_{p})(\varphi)\quad \mbox{\rm with}\quad 
\Phi^{n-p}(\eta_p)(\varphi)=\Phi\circ\Phi\circ\ldots\circ \Phi(\eta_p)
$$
where the composition is $(n-p)-$times.
Note that
$$
\Phi^{n-p}(\eta_{p})(\varphi)=\frac{\eta_p Q^{n-p}(\varphi)}{\eta_p Q^{n-p}(1)}\quad
\mbox{\rm and}\quad \Phi^l(\overline{\eta}_n)(\varphi)=
\frac{\overline{\eta}_nQ^l(\varphi)}{\overline{\eta}_nQ^l(1)}.
$$
For any given $l\geq 1$ we set
$$
\mu^{l}_n(\varphi):=\frac{1}{n}\sum_{0\leq k< n}\gamma_{l+k}(\varphi)/\gamma_{k}(1)=
\frac{1}{n}\sum_{0\leq k< n}\eta_k(Q^l(\varphi))=\overline{\eta}_nQ^l(\varphi)
$$
and their normalized versions
$$
\overline{\mu}^{l}_n(\varphi):=
\mu^{l}_n(\varphi)/\mu^{l}_n(1)\Longrightarrow \overline{\mu}^{l}_n= \Phi^l(\overline{\eta}_n).
$$
In our context, under some stability conditions (see e.g.~\cite{delm,whiteley} for instance),  there exists an invariant measure
$\eta_{\infty}$, a parameter $\lambda>0$ and a function $h\geq 0$ such that for any $(x,\varphi)\in\mathsf{E}\times\mathcal{B}_b(\mathsf{E})$
$$
Q(h)(x)=\lambda~h(x)\quad\mbox{\rm and}\quad
\eta_{\infty}(\varphi)=\Phi(\eta_{\infty})(\varphi)\quad (\Longrightarrow\quad Q^n(h)=\lambda^n~h).
$$
Applying the above fixed point equation to $\varphi=h$ we check that
$$
\Phi^n(\eta_{\infty})(h)=\frac{\eta_{\infty}Q^n(h)}{\eta_{\infty}Q^n(1)}=\eta_{\infty}(h)=\lambda^n~\frac{\eta_{\infty}(h)}{\eta_{\infty}Q^n(1)}\Longleftrightarrow
\eta_{\infty}Q^n(1)=\lambda^n.
$$
Also note that
$$
G(x)=Q(1)(x)\Longrightarrow \eta_{\infty}(G)=\lambda.
$$
More generally, for any $\varphi\in\mathcal{B}_b(\mathsf{E})$ we have
$$
\Phi^n(\eta_{\infty})(\varphi)=\eta_{\infty}(\varphi)\Longleftrightarrow
\eta_{\infty}Q^n(\varphi)=\lambda^n~\eta_{\infty}(\varphi).
$$

\subsection{Diffusion Monte Carlo}

DMC is a discrete-time system of $N$ walkers $\xi_n=\left(\xi_n^i\right)_{1\leq i\leq N}$. The system starts with $N$ independent copies of a random variable with distribution $\eta_0$.  Given the system   $\xi_{n}$ at some time $n\geq 0$, we sample $N$ conditionally independent walkers  $\xi_{n+1}^i$ with their respective distribution
$
\Phi(\eta_n^N)
$,
with  $$\eta^N_n:=m(\xi_n):=\frac{1}{N}\sum_{1\leq i\leq N}\delta_{\xi^i_n}.
$$

In other words, the DMC method consists of approximating the measure $\eta_n$ by using the occupation measure $\eta^N_n$ associated with a system of $N$ walkers. The initial positions of the walkers are randomly chosen from the distribution $\eta_0$. The evolution of each walker follows then the following selection/mutation steps:
\begin{itemize}
\item Selection: We evaluate the current position $\xi_n^i$ of all of the walkers at its potential value $G(\xi_n^i)$. 
For each walker,  we select a walker amongst the current collection with probability $G(\xi_n^i)/\eta_n^N(G)$, giving the new position $\widehat{\xi}_n^i$.
\item Mutation: We move the selected walker $\widehat{\xi}_n^i=x$ to a new location $\xi_{n+1}^i=y$ using the transition kernel $M(x,dy)$.
\end{itemize}

\section{Numerical Estimator}\label{sec:est}

\subsection{Unnormalized Particle Measures}

We consider some well-known facts about particle approximation of unnormalized measures; this discussion will prove useful as we give lagged estimators that are used in the physics literature.
Mimicking the r.h.s.~formula in \eqref{def-gamma}, the unbiased estimate of $\gamma_n(\varphi)$ (see \cite[Chapter 7]{delm}) is defined  by
$$
\gamma^N_n(\varphi):=\eta^N_n(\varphi)~\prod_{0\leq p<n}\eta^N_p(G).
$$
The unbiasedness property ensures that
$$
\EE(\gamma^N_n(\varphi))=\gamma_n(\varphi)=\eta_0Q^{n}(\varphi).
$$ 
To check this claim, note that
$$
\EE(\eta^N_n(\varphi)~|~\xi_{n-1})=\Phi(\eta_{n-1}^N)(\varphi)=\frac{\eta_{n-1}^NQ(\varphi)}{\eta_{n-1}^NQ(1)}=\frac{\eta_{n-1}^NQ(\varphi)}{\eta_{n-1}^N(G)}.
$$
This implies that
$$
\EE(\gamma^N_n(\varphi)~|~\xi_{n-1})=\eta_{n-1}^NQ(\varphi)~\prod_{0\leq p<(n-1)}\eta^N_p(G)=
\gamma^N_{n-1}Q(\varphi).
$$
Iterating the argument, for any $k< n$ we check that
\begin{equation}\label{eq:cond_exp_gamma}
\EE(\gamma^N_n(\varphi)~|~\xi_{k})=\gamma^N_{k}Q^{n-k}(\varphi).
\end{equation}
Applying the above to $k=0$ we recover the unbiasedness property
$$
\EE(\gamma^N_n(\varphi)~|~\xi_{0})=\eta^N_{0}Q^{n}(\varphi)\quad \mbox{\rm and}\quad
\EE(\gamma^N_n(\varphi))=\eta_{0}Q^{n}(\varphi)=\gamma_n(\varphi).
$$
Note that 
$$
\eta_0=\eta_{\infty}\Longrightarrow
\EE(\gamma^N_n(\varphi))=\eta_{\infty}(Q^n(\varphi))=\lambda^n~\eta_{\infty}(\varphi).
$$

\subsection{Fix Lagged Estimator}

For any  given $(l,k,\varphi)\in\mathbb{N}_0^2\times\mathcal{B}_b(\mathsf{E})$ we have
$$
\gamma^N_{l+k}(\varphi)/\gamma^N_{k}(1)=\eta^N_{l+k}(\varphi)~\prod_{k\leq p<k+l}\eta^N_p(G)=F_{\varphi}(\xi_k,\xi_{k+1},\ldots,\xi_{k+l})
$$
with
$$
F_{\varphi}(\xi_k,\xi_{k+1},\ldots,\xi_{k+l})=
m(\xi_{l+k})(\varphi)~\prod_{k\leq p<k+l} m(\xi_p)(G).
$$
The estimate we want to analyze is based upon:
$$
\mu^{l,N}_n(\varphi):=\frac{1}{n}\sum_{0\leq k< n}\gamma^N_{l+k}(\varphi)/\gamma^N_{k}(1)=
\frac{1}{n}\sum_{0\leq k< n}F_\varphi(\xi_k,\xi_{k+1},\ldots,\xi_{k+l}).
$$
The estimator is:
\begin{equation}\label{eq:main_est}
\overline{\mu}^{l,N}_n(\varphi):=\frac{\mu^{l,N}_n(\varphi)}{\mu^{l,N}_n(1)}=\frac{\frac{1}{n}\sum_{0\leq k< n}F_{\varphi}(\xi_k,\xi_{k+1},\ldots,\xi_{k+l})}{\frac{1}{n}\sum_{0\leq k< n}F_1(\xi_k,\xi_{k+1},\ldots,\xi_{k+l})}.
\end{equation}

To understand the intuition for this estimator, observe that
$$
l=0\Longrightarrow
\gamma^N_{k}(\varphi)/\gamma^N_{k}(1)=\eta^N_k(\varphi)
$$
and more generally for any $l\geq 0$ we have, noting \eqref{eq:cond_exp_gamma}:
$$
\frac{
\EE\left(
\gamma^N_{k+l}(\varphi)~|~\xi_{k}\right)}{\gamma^N_{k}(1)}=
\frac{\gamma^N_{k}Q^{l}(\varphi)}{\gamma^N_{k}(1)}=\eta^N_{k}Q^{l}(\varphi)
$$
and hence
$$
\EE(\mu^{l,N}_n(\varphi))=\frac{1}{n}\sum_{0\leq k< n}\EE\left(\gamma^N_{l+k}(\varphi)/\gamma^N_{k}(1)\right)=
\frac{1}{n}\sum_{0\leq k< n}\EE\left(\eta^N_{k}Q^{l}(\varphi)\right).
$$

The limiting object is given by
$$
\frac{1}{n}\sum_{0\leq k< n}\eta_{k}(Q^l(\varphi))\simeq_{n\uparrow\infty}
\eta_{\infty}(Q^l(\varphi))=\lambda^l~\eta_{\infty}(\varphi).
$$
As a result,  we expect that for long-time intervals $n$ that
$\overline{\mu}^{l,N}_n(G)$ should give a good approximation of
$\lambda$ the eigenvalue; or indeed,  $\overline{\mu}^{l,N}_n(\varphi)$ an approximation of $\eta_{\infty}(\varphi)$. 
What is important, is to understand the bias and variance of this estimator, precisely in the long-time regime, which is the topic of the next section. Particularly,  what the effect of $l\in\mathbb{N}$ is.

%

\section{Main Results}\label{sec:main_res}

\subsection{Assumption}

\begin{hypA}\label{ass:smc}
There exists a $C<+\infty$ such that
$$
\sup_{(x,y)\in\mathsf{E}^2}\frac{G(x)}{G(y)} \leq C.
$$
There exists a $(\theta,\beta)\in(0,1)\times\mathcal{P}(\mathsf{E})$ such that for any $(x,A)\in\mathsf{E}\times\mathcal{E}$
$$
\theta\int_A \beta(dy) \leq \int_{A}M(x,dy) \leq \theta^{-1}\int_A \beta(dy).
$$ 
\end{hypA}

This assumption is well-studied in the context of particle approximations of Feynman-Kac formulae; see
\cite{delm} for example.  The assumption will typically hold in scenarios where the state-space $\mathsf{E}$ is compact and seldom holds otherwise.  
For instance, it is satisfied for Markov transitions of 
elliptic diffusions on compact manifolds $\mathsf{E}$, see for instance the pioneering work of \cite{aronson,nash,varopoulos}
on Gaussian estimates for heat kernels on manifolds. 
The assumption can be relaxed using the methods in \cite{beskos,caffareldp-24,dhj,douc,whiteley1,whiteley3} at the cost of more technical proofs that one would expect to lead to the same qualitative conclusions.
Note that (A\ref{ass:smc}) assures that  for $\varphi\in\mathcal{B}_b(\mathsf{E})$
$$
\eta_{\infty}(\varphi) := \lim_{n\rightarrow\infty}\frac{\gamma_n(\varphi)}{\gamma_n(1)}.
$$
is well-defined.  Indeed in \cite{delm,delm:13,dd04,dm_2003,whiteley} rates of convergence of $\eta_n(\varphi)$ to 
$\eta_{\infty}(\varphi)$ have been established.

\subsection{Particle System as a Markov Chain}\label{sec:particle_mc}

We can consider the particle system $\xi_{0},\xi_1,\dots$ as a time-homogeneous Markov chain on $(\mathsf{E}^N,\mathcal{E}^{\otimes N})$ with initial distribution $\prod_{i=1}^N \eta_0(d\xi_0^i)$ and transition 
$R^N(\xi_{p-1},d\xi_p)=\prod_{i=1}^N\Phi(m(\xi_{p-1}))(d\xi_p^i)$. 
Moreover, under (A\ref{ass:smc}) the Markov chain is uniformly ergodic and there exist a unique invariant measure,  which we shall call $\Pi^N$.  
Note that it is simple to see that $\Pi^N$ is deterministic,  in that it does not depend on the randomness of the simulated particle system and,  indeed,  an exact expression is availabe; see \cite[Corollary 1]{hobert} 
for example.  Now as the particle system is exchangeable,  the marginal of $\Pi^N$ in any co-ordinate is the same
and we will write it as $\pi^{N}\in\mathcal{P}(\mathsf{E})$.  We have the following result which shall prove useful
in the subsequent discussion.

\begin{prop}\label{prop:inv_dist}
Assume (A\ref{ass:smc}).  Then there exists a $C\in(0,\infty)$ such that for any $N\in\mathbb{N}$:
$$
\Vert \pi^N-\eta_{\infty}\Vert_{\tiny \textrm{\emph{tv}}}\leq \frac{C}{N}.
$$
\end{prop}

\begin{proof}
Recall from~\cite{dg98,delmoral_2000} the time-uniform estimate, for $\varphi\in\mathcal{B}_b(\mathsf{E})$
\begin{equation}\label{eq:bias_uniform}
\sup_{n\geq 0}\left|\EE\left[\eta^N_n(\varphi)-\eta_n(\varphi)\right]\right| \leq \frac{C\|\varphi\|}{N}.
\end{equation}
Now we have that
\begin{eqnarray*}
\left|\EE\left[\varphi(\xi_n^1)-\eta_{\infty}(\varphi)\right]\right| & = &  \left|\EE\left[\eta^N_n(\varphi)-\eta_{\infty}(\varphi)\right]\right| \\
&\leq &
\left|\EE\left[\eta^N_n(\varphi)-\eta_{n}(\varphi)\right]\right| + |\eta_{n}(\varphi)-\eta_{\infty}(\varphi)|.
\end{eqnarray*}
Then on noting that
$$
|\pi^N(\varphi)-\eta_{\infty}(\varphi)| = \lim_{n\rightarrow+\infty}\left|\EE\left[\varphi(\xi_n^1)-\eta_{\infty}(\varphi)\right]\right|
$$
the result is concluded by using \eqref{eq:bias_uniform},  $\lim_{n\rightarrow+\infty}|\eta_{n}(\varphi)-\eta_{\infty}(\varphi)|=0$ and basic properties of the total variation distance.
\end{proof}

\subsection{The Effect of the Lag}

\begin{theorem}\label{prop:main_prop}
Assume (A\ref{ass:smc}).  Then for any
$(N,l,\varphi)\in\mathbb{N}^2\times\mathcal{B}_b(\mathsf{E})$ we have almost surely:
$$
\lim_{n\rightarrow+\infty}
\overline{\mu}^{l,N}_n(\varphi)
-\eta_{\infty}(\varphi) = \Phi^l(\pi^{N})(\varphi)- \Phi^l(\eta_{\infty})(\varphi)
$$
\end{theorem}

\begin{proof}
Using the exposition in Section \ref{sec:particle_mc} we can use 
Proposition \ref{prop:mc} in the Appendix along with a simple first Borel-Cantelli argument,  to determine that 
\begin{equation}\label{eq:conv_as}
\mu^{l,N}_n(\varphi)\rightarrow_{a.s.}
\int_{\mathsf{E}^{(l+1)N}}m(\xi_{l+1})(\varphi)
\left\{\prod_{p=1}^{l-1}m(\xi_{p})(G)\right\}
\Pi^N(d\xi_1) \prod_{p=1}^{l+1}R^N(\xi_{p-1},d\xi_p).
\end{equation}
where $\rightarrow_{a.s.}$ denotes almost sure convergence as $n\rightarrow\infty$. Moreover, by using standard properties of particle filters,  (see e.g.~\cite[Chapter 7]{delm}) the R.H.S.~of \eqref{eq:conv_as} is equal to
$$
\Pi^N\left(m\left(Q^l(\varphi)\right)\right)
$$
where $m$ is the $N-$equally weighted empirical measure.
Therefore we one has that
\begin{equation}\label{eq:prob_c}
\overline{\mu}^{l,N}_n(\varphi)
\rightarrow_{a.s.}
 \frac{\Pi^N\left(m\left(Q^l(\varphi)\right)\right)}{\Pi^N\left(m\left(Q^l(1)\right)\right)}.
\end{equation}
As the particle system is exchangeable,  the marginal of $\Pi^N$ in any co-ordinate is the same,
we have that 
$$
 \frac{\Pi^N\left(m\left(Q^l(\varphi)\right)\right)}{\Pi^N\left(m\left(Q^l(1)\right)\right)} = 
 \frac{\pi^N\left(Q^l(\varphi)\right)}{\pi^N\left(Q^l(1)\right)} = \Phi^l(\pi^N)(\varphi).
$$
As
$$
\eta_{\infty}(\varphi) =  \Phi^l(\eta_{\infty})\left(\varphi\right)
$$
the proof is completed.
\end{proof}

\begin{cor}\label{cor:main_cor}
Assume (A\ref{ass:smc}).  Then 
there exists a $(C,\kappa)\in(0,\infty)^2$ such that 
for any
$(N,l,\varphi)\in\mathbb{N}^2\times\mathcal{B}_b(\mathsf{E})$ we have almost surely:
$$
\lim_{n\rightarrow+\infty}|\overline{\mu}^{l,N}_n(\varphi)-\eta_{\infty}(\varphi)| \leq \frac{Ce^{-\kappa l}\|\varphi\|}{N}.
$$
\end{cor}

\begin{proof}
Using Theorem \ref{prop:main_prop}  the proof can be completed by using the exponential stability property of the semigroup $\Phi(\cdot)$ which holds under 
(A\ref{ass:smc}); see \cite[Chapter 4]{delm} for example.  More precisely,  we have the exponential decay
$$
\Vert \Phi^l(\pi^{N})- \Phi^l(\eta_{\infty})\Vert_{\tiny \textrm{tv}}\leq C e^{-\kappa l}~\Vert \pi^N-\eta_{\infty}\Vert_{\tiny \textrm{tv}}
$$
for some constants $(C,\kappa)\in(0,\infty)^2$ that do not depend upon $(N,l,\varphi)$.
Application of Proposition \ref{prop:inv_dist} allows one to conclude.
\end{proof}

\begin{rem}
The approach here relies on a novel path-based Markov chain $\mathbb{L}_p-$bound using a martingale plus remainer structure. It is not the first theoretical analysis for particle methods that keeps $N$ fixed and allows another parameter to grow; see \cite{beskos,whiteley2}.
\end{rem}

The implication of Theorem \ref{prop:main_prop} and Corollary \ref{cor:main_cor} is rather interesting.  
The results say that for the asymptotic in $n$ regime, that 
the almost sure $\mathbb{L}_1-$error is exponentially small as a function of $l$ 
\emph{irregardless of the number of walkers}, that is one can take $N=\mathcal{O}(1)$.  
From a practical point of view it suggests that a small number of walkers can be run for a long time and so long as $l$ is moderate,  the error of the estimator should be small.  This shows that the parameter $l$ provides quite a substantial freedom.  For instance, one could estimate $\lambda$ using the estimator:
\begin{equation}\label{eq:other_est}
\overline{\eta}^N_n(G):=\frac{1}{n}\sum_{k=0}^{n-1}\eta_k^N(G)
\end{equation}
for which, using the approaches in \cite{delm} and under mathematical assumptions would have a bias 
$$
\left|\mathbb{E}\left[\frac{1}{n}\sum_{k=0}^{n-1}\left\{\eta_k^N(G)-\eta_k(G)\right\}\right]\right|
$$
of $\mathcal{O}(N^{-1})$ for any fixed $n$.  Then it is simple to establish that the asymptotic in $n$ bias
is upper-bounded by a term that is $\mathcal{O}(N^{-1})$.  What this means is that 
the only way in which this conventional estimate can indeed recover $\lambda$, for long-time periods,  is by increasing the number of walkers, which could be subtantially more expensive than using the estimator \eqref{eq:main_est}.  More formally,  when considering the estimator \eqref{eq:main_est} ,  for $\epsilon\in(0,1)$ given,   asymptotically in $n$,  to obtain an error of $\mathcal{O}(\epsilon)$,  one can choose $l=\mathcal{O}(-\log(\epsilon))$, $N=\mathcal{O}(1)$. When considering the estimate \eqref{eq:other_est} then $N=\mathcal{O}(\epsilon^{-1})$ to have an asymptotic bias of $\mathcal{O}(\epsilon)$.  The cost of this computation could be substantially more than considering
\eqref{eq:main_est} depending on the relative of cost of computing
\eqref{eq:main_est} over \eqref{eq:other_est}.


\subsection{Non-Asymptotic Bounds}

We now investigate the non-asymptotic $\mathbb{L}_1-$error,  where $n,N,l$ are all fixed.
Below $\mathbb{N}_0=\mathbb{N}\cup\{0\}$. We have the following, uniform in time, non-asymptotic
error bound.

\begin{theorem}\label{theo:l1}
Assume (A\ref{ass:smc}). Then there exists a $C\in(0,\infty)$ such that for any $(n,N,l,\varphi)\in
\mathbb{N}_0\times
\mathbb{N}^2\times\mathcal{B}_b(\mathsf{E})$:
$$
\mathbb{E}\left[\left|
\overline{\mu}^{l,N}_n(\varphi)
-
\overline{\mu}^{l}_n(\varphi)
\right|\right] \leq \frac{C\|\varphi\|l}{\sqrt{N}}.
$$
\end{theorem}

\begin{proof}
We have:
$$
\overline{\mu}^{l,N}_n(\varphi)
-
\overline{\mu}^{l}_n(\varphi)
= T_1 + T_2
$$
where
\begin{eqnarray*}
T_1 & = & \frac{\mu^{l,N}_n(\varphi)}{\mu^{l,N}_n(1)\mu^{l}_n(1)}\left(\mu^{l}_n(1)-\mu^{l,N}_n(1)\right) \\
T_2 & = & \frac{1}{\mu^{l}_n(1)}\left(\mu^{l,N}_n(\varphi)-\mu^{l}_n(\varphi)\right).
\end{eqnarray*}
The proof is easily completed by using Minkowski, the fact that
$$
\frac{\mu^{l,N}_n(\varphi)}{\mu^{l,N}_n(1)} \leq \|\varphi\|
$$
and Lemma \ref{tech:lem}.
\end{proof}

\begin{rem}\label{rem:bias}
If one wants to consider $\mathbb{E}\left[\left|
\overline{\mu}^{l,N}_n(\varphi)-
\eta_{\infty}(\varphi)\right|\right]$ then by using results on the rates of the bias (see \cite{dd04,whiteley})
one would,  under (A\ref{ass:smc}),  have an upper-bound of the type:
$$
\mathcal{O}\left(\frac{l}{\sqrt{N}} + \kappa^n\right)
$$
for $\kappa\in(0,1)$.
\end{rem}

Theorem \ref{theo:l1} shows that one cannot naively increase $l$ so as to obtain a small error (as in 
Corollary \ref{cor:main_cor}) and there is at most a linear increase in the non-asymptotic $\mathbb{L}_1-$error. On the basis
of Corollary \ref{cor:main_cor} and Theorem \ref{theo:l1} one expects that the asymptotic (in $n$) $\mathbb{L}_1-$error
is at most
$$
\mathcal{O}\left(\min\left\{\frac{e^{-\kappa l}}{N},\frac{l}{\sqrt{N}}\right\}\right)
$$
so that again one can choose $l=\mathcal{O}(-\log(\epsilon))$, $N=\mathcal{O}(1)$ to obtain the error
as $\mathcal{O}(\epsilon)$ ($\epsilon\in(0,1)$). 
That is,  there is not a contradiction with our previous discussion.  None-the-less,
the afore-mentioned points are based upon asymptotics in $n$.  For instance,  combining 
Theorem \ref{theo:l1} and Remark \ref{rem:bias} choose $n=\mathcal{O}(-\log(\epsilon))$,  $l=f(\epsilon)$,  where
$f$ is a non-decreasing function which explodes at zero and $N=\mathcal{O}(\epsilon^{-2}f(\epsilon)^2)$ would achieve a non-asymptotic $\mathbb{L}_1-$error of $\mathcal{O}(\epsilon)$.  This again establishes the flexibility of the estimator where one can of course control the error with $N$ if needed.

\subsection{Central Limit Theorem}

Whilst the previous results give a characterization of the bias and variance for large $n$, one can give exact expressions of this asymptotic error.  This is done in the following central limit theorem.

We require several notations,  which are given now.  For 
$(\xi_0,\dots,\xi_{l},\varphi)\in\mathsf{E}^{N(l+1)}\times\mathcal{B}_b(\mathsf{E})$ define:
\begin{equation}\label{eq:pois_main}
\widehat{F}_{\varphi}(\xi_0,\dots,\xi_{l})  := F_{\varphi}(\xi_0,\dots,\xi_{l}) - \pi^{N}(Q^l(\varphi)) 
+ \sum_{q=1}^{\infty} \left\{\mathbb{E}\left[\Phi(\eta_{(q-1)l}^N)(Q^l(\varphi))|\Xi_0=\xi_{l}\right]- \pi^{N}(Q^l(\varphi))\right\}
\end{equation}
where $\mathbb{E}[\cdot|\Xi_0=\xi_{l}]$ is the expectation w.r.t.~the law associated to the DMC algorithm,  with the initial particles at time zero equal to $\xi_{l}$.  Using (A\ref{ass:smc}) one can show that the function $\widehat{F}_{\varphi}$ is upper-bounded by a constant (which would explode as $N$ grows).
Now define for $(q,r)\in\{1,\dots,n\}\times\{0,\dots,l\}$
\begin{eqnarray}
(R^N)^{\otimes(l-r)}(\widehat{F}_\varphi)(\xi_{q},\dots,\xi_{q+r}) & := &  \int_{\mathsf{E}^{N(l-r)}}\widehat{F}_{\varphi}(\xi_q,\dots,\xi_{q+l})
\prod_{j=q+r+1}^{q+l}R^N(\xi_{j-1},d\xi_j)\label{eq:mainprod1}\\
(R^N)^{\otimes(l+1-r)}(\widehat{\varphi})(\xi_q,\dots,\xi_{q+r-1}) & := &  \int_{\mathsf{E}^{N(l+1-r)}}\widehat{F}_{\varphi}(\xi_q,\dots,\xi_{q+l})
\prod_{j=q+r}^{q+l}R^N(\xi_{j-1},d\xi_j)\label{eq:mainprod2}
\end{eqnarray}
with the convention that $(\xi_{q},\xi_{q-1}) = \xi_{q-1}$.  We will also write 
\begin{equation}
(R^N)^{\otimes(0)}(\widehat{\varphi})(\xi_q,\dots,\xi_{q+l}) = \widehat{F}_{\varphi}(\xi_q,\dots,\xi_{q+l}).\label{eq:mainprod3}
\end{equation}
Next, for $(l,k,\varphi)\in\mathbb{N}\times\{l,l+1,\dots\}\times\mathcal{B}_b(\mathsf{E})$
\begin{align}
\sigma^2_{F_{\varphi}} & :=  \int_{\mathsf{E}^{Nl}} \tilde{F}_{\varphi}(\xi_{k-l},\dots,\xi_{k-1})
\Pi^N(d\xi_{k-l})
\prod_{s=k-l+1}^{k-1} R^N(\xi_{s-1},d\xi_s) \label{eq:mainclt_av1}
\end{align}
\begin{align}
\tilde{F}_{\varphi}(\xi_{k-l},\dots,\xi_{k-1})
 & =  \int_{\mathsf{E}^N}\left(\sum_{j=0}^{l} (R^N)^{\otimes(l-j)}(\widehat{F}_{\varphi})(
\xi_{k-j},\dots,\xi_k)\right)^2 R^N(\xi_{k-1},d\xi_k) - \nonumber\\
& \left(\sum_{j=0}^{l} (R^N)^{\otimes(l+1-j)}(\widehat{F}_{\varphi})(\xi_{k-j},\dots,\xi_{k-1})
\right)^2.\label{eq:mainclt_av2}
\end{align}
Finally set
\begin{align*}
\overline{F}_{\varphi}(\xi_0,\dots,\xi_l) & := \frac{1}{\pi^{N}(Q^l(1))}F_{\varphi}(\xi_0,\dots,\xi_l) - 
\frac{\pi^{N}(Q^l(\varphi))}{\pi^{N}(Q^l(1))^2}F_{1}(\xi_0,\dots,\xi_l) \\
\widehat{\overline{F}}_{\varphi}(\xi_0,\dots,\xi_{l}) & :=  \frac{1}{\pi^{N}(Q^l(1))}\widehat{F}_{\varphi}(\xi_0,\dots,\xi_l) - 
\frac{\pi^{N}(Q^l(\varphi))}{\pi^{N}(Q^l(1))^2}\widehat{F}_{1}(\xi_0,\dots,\xi_l).
\end{align*}
Below we use $\xrightarrow[]{d}$ to denote convergence in distribution as $n$ increases. 

\begin{theorem}\label{theo:clt1}
    Assume (A\ref{ass:smc}). For any $(N,l,\varphi)\in\mathbb{N}^2\times\mathcal{B}_b(\mathsf{E})$ we have
\begin{equation*}
\sqrt{n}\left(\overline{\mu}^{l,N}_n(\varphi)
-
\Phi^l(\pi^{N})\left(\varphi\right)
\right) 
\xrightarrow[]{d} \mathcal{N}(0, \sigma_{\overline{F}_{\varphi}}^2).
\end{equation*}
\end{theorem}
\begin{proof}
This is Proposition \ref{prop:clt_2} in the appendix,  amended to the notation used in the main text.  For instance,  one can compare \eqref{eq:mainprod1}-\eqref{eq:mainprod3} with \eqref{eq:prod1}-\eqref{eq:prod3} and
\eqref{eq:mainclt_av1}-\eqref{eq:mainclt_av2} with \eqref{eq:clt_av1}-\eqref{eq:clt_av2}.
\end{proof}

\begin{rem}
We note that 
$$
\sqrt{n}\left|\Phi^l(\pi^{N})\left(\varphi\right)-\eta_{\infty}(\varphi)\right|
\leq \frac{C\sqrt{n}\exp\{-\kappa l\}}{N}
$$
which means one must choose $l$ and $N$ appropriately to control $\sqrt{n}\left|\Phi^l(\pi^{N})\left(\varphi\right)-\eta_{\infty}(\varphi)\right|$.  For instance
one could have $n=\mathcal{O}(\epsilon^{-2})$, $l=\mathcal{O}(-\log(\epsilon))$ and $N=\mathcal{O}(1)$
to make this latter term small for large $n$.
\end{rem}

Theorem \ref{theo:clt1} gives an \emph{exact} description of the errors of the estimator in $N$ and $l$ through the expression of the asymptotic variance $\sigma_{\overline{F}_{\varphi}}^2$.  As is typical in Markov chain CLTs,  this latter variance is written in terms of a solution to the Poisson equation (see e.g.~\cite{glynn}) as in 
\eqref{eq:pois_main}.  $\sigma_{\overline{F}_{\varphi}}^2$ is further complicated as one has the form of a ratio of lagged Markov chain estimators.  In general one would need numerical methods to approximate the asymptotic variance and this is considered below.  
Understanding how $\sigma_{\overline{F}_{\varphi}}^2$ behaves as a function of $N,l$ is an important mathematical question,  which requires further investigation, but one might expect that it is related to
Theorem \ref{theo:l1} and this is then $\mathcal{O}(l^2/N)$.

To conclude this section on CLTs one can also give a result related to an approach that is considered in the literature.  In some cases,  authors have proposed running two different (independent) DMC algorithms, one that gives the estimate 
$\mu^{l,N}_n(\varphi)$ and the other which gives the estimate $\mu^{l,N}_n(1)$. In this case defining
$$
\sigma_{\varphi}^{2,\textrm{ind}} = \frac{1}{\pi^{N}(Q^l(1))}\sigma^2_{F_{\varphi}}
+\frac{\pi^{N}(Q^l(\varphi))^2}{\pi^{N}(Q^l(1))^4} \sigma^2_{F_{1}}
$$
one can apply Proposition \ref{prop:clt_3} in the appendix to show that the asymptotic variance for such estimators (under (A\ref{ass:smc}))
is exactly $\sigma_{\varphi}^{2,\textrm{ind}}$.  In general,  one expects that $\sigma_{\overline{F}_{\varphi}}^2<\sigma_{\varphi}^{2,\textrm{ind}}$
but proving this is rather difficult.  We consider a numerical comparison below.

\subsection{Numerical Study}

Consider the quantum harmonic oscillator Hamiltonian
$$H=-\frac{1}{2m}\nabla^2 + \frac{1}{2}m\omega^2x^2.$$
The ground state energy (smallest eigenvalue) of the operator $H$ is known to be $E_0=\left(n+\frac{1}{2}\right)\omega$. Following the discussion in Appendix A, we set $\tau=1/16$ (time discretization), $\omega=1$, and $m=1$. Define $G(x)=e^{-\tau x^2/2}$ and $M(\cdot,dx)$ to be the Brownian motion transition kernel (over time unit $\tau$).  The operator $Q$ is given by $Q=e^{-\tau H}$ and the largest eigenvalue of the operator $Q$ is given by $\lambda = e^{-\tau E_0}$.

We ran the DMC algorithm and calculated the fixed lag estimator\eqref{eq:main_est} with $\varphi=G$. We varied the lag values from $0$ to $50$, with $N=10$ and $n=50000$.  In order to calculate the bias and the variance we conducted $128$ independent runs.  For each lag value, we obtain an estimator of $\lambda$.
Lag $0$ case corresponds to the standard method the DMC particles are used to estimate $\lambda$. Non-zero lag values correspond to the estimators studied in this paper.  Corollary \ref{cor:main_cor} predicts that for a fixed number of particles $N$ and a large $n$, the bias will decay at an exponential rate.  Figure \ref{fig:bias} presents two plots, one for the absolute bias and one for the log absolute bias each plotted against the different lag values. The figure exhibits a linear trend in the log absolute bias plot which affirms the statement about the exponential decaying bias stated in Corollary \ref{cor:main_cor}.

\begin{figure}[h!]
    \centering
    \includegraphics[width=0.45\linewidth]{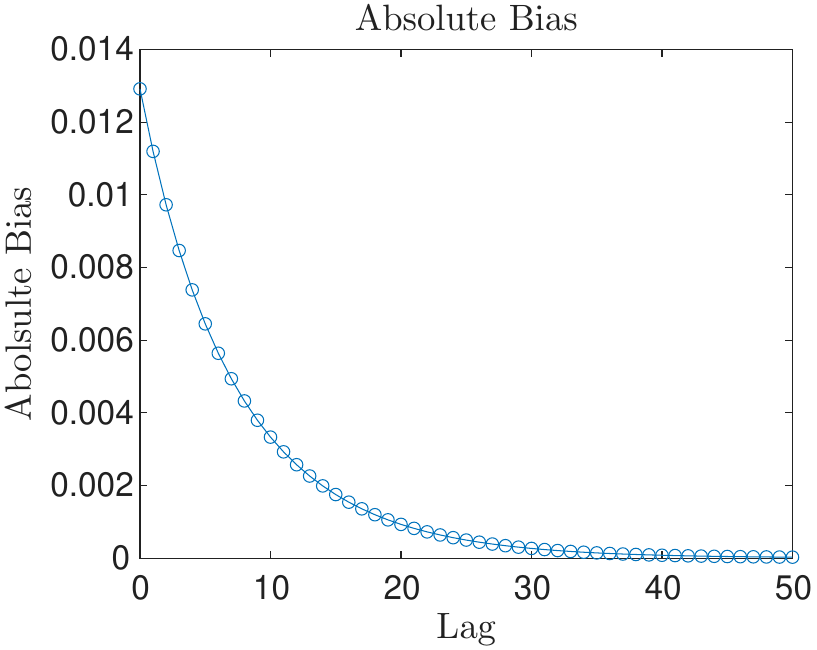}
    \includegraphics[width=0.45\linewidth]{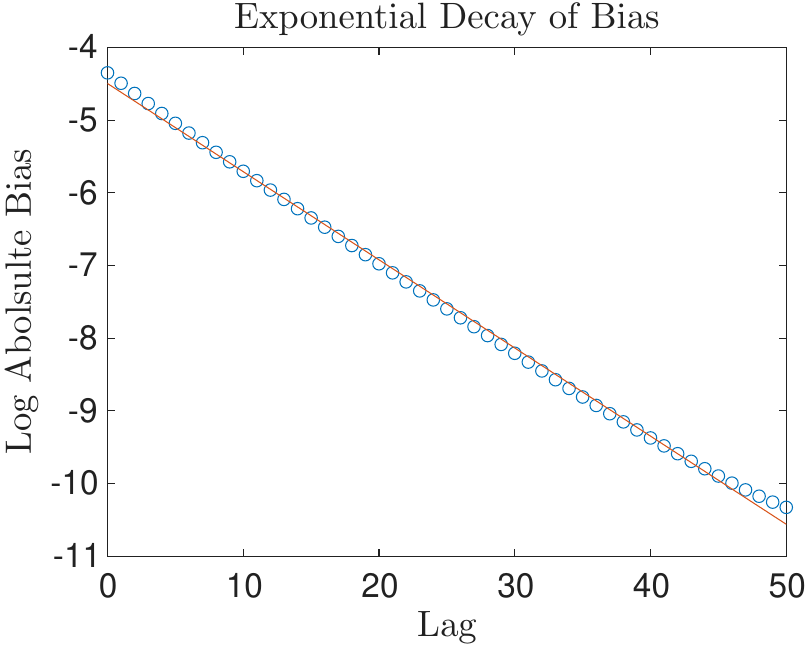}
    \caption{Bias Comparison.  Left: Absolute Bias vs Lag. Right: Log Aboslute Bias vs Lag.}
    \label{fig:bias}
\end{figure}

Let $(\xi_i)_{i\geq0}$ and $(\tilde{\xi_i})_{i\geq0}$ be identically distributed and independent DMC particles. Consider the following two estimators:
\begin{equation*}
    \bar{\mu}_{n}^{l,N} = \frac{\frac{1}{n}\sum_{0\leq k\leq n}F_{G}(\xi_k,\dots,\xi_{k+l})}{\frac{1}{n}\sum_{0\leq k\leq n}F_{1}(\xi_k,\dots,\xi_{k+l})}, \quad\quad
    \tilde{\mu}_{n}^{l,N} = \frac{\frac{1}{n}\sum_{0\leq k\leq n}F_{G}(\xi_k,\dots,\xi_{k+l})}{\frac{1}{n}\sum_{0\leq k\leq n}F_{1}(\tilde{\xi}_k,\dots,\tilde{\xi}_{k+l})}.
\end{equation*}
Figure \ref{fig:var} shows the variance of both estimators under the same setting used for the bias simulation above. The simulation results indicate that the variance of the estimator $\bar{\mu}_n^{l,N}$ is much smaller than the variance of $\tilde{\mu}_n^{l,N}$. This is because we expect that the numerator and denominator of $\bar{\mu}_n^{l,N}$ are highly positively correlated which reduces the overall variance of the estimator $\bar{\mu}_n^{l,N}$. Moreover, comparing Figures \ref{fig:bias} and \ref{fig:var} reveals an opposing trend between the bias and the variance as the lag value increases. The bias exponentially decreases, while the variance increases.

\begin{figure}[h!]
    \centering
    \includegraphics[width=0.45\linewidth]{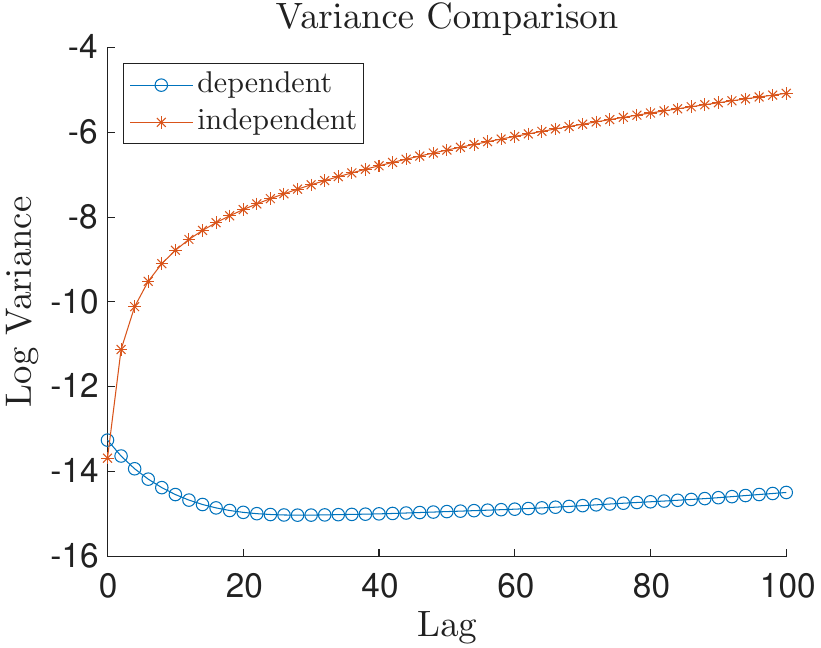}
    \caption{Variance Comparison.}
    \label{fig:var}
\end{figure}

\subsubsection*{Acknowledgements}

AJ was supported by CUHK-SZ UDF01003537.
MC thanks the European Research Council (ERC)
under the European Union's Horizon 2020 research and
innovation programme (grant agreement no. 863481) for
financial support.

\appendix

\section{Feynman-Kac formula in physics}
\label{FKP}
In physics we are interested in solving the imaginary-time dependent Sch\"odinger equation
\be
\frac{ \partial \psi({\bf x},t) }{\partial t} = - H \psi({\bf x},t)
\label{eq1}
\ee
with some function $\psi({\bf x},0)$ as initial condition and ${\bf x} \in \mathsf{E}={\mathbb{R}}^d$ with $d \ge 1$. 
Here, $H$ is the Schr\"odinger Hamiltonian given by
$$
H = -\frac{1}{2} \nabla^2 + V({\bf x}).
$$
The solution of Eq.(\ref{eq1}) is expressed using the Feynman-Kac formula as follows
$$
\psi({\bf x},t) =\mathbb{E}_{{\bf x}}[e^{-\int_0^t ds V({\bf X}(s))}]
$$
where the expectation $\mathbb{E}_{{\bf x}}$ is with respect to the law of the Brownian process with initial distribution $\psi({\bf x},0)$ and 
final condition ${\bf X}(t)={\bf x}$.
Using the notations of the present article, this FK formula can be put into the form 
$$
\gamma_t(\phi) \equiv \int d{\bf x} \phi({\bf x}) \psi({\bf x},t) = \mathbb{E}[  \phi({\bf X}(t)) e^{-\int_0^t ds V({\bf X}(s))}].
$$
In DMC the continuous time variable needs to be discretized. Writing $t= n\tau$ where $\tau$ is a time-step (a "small" positive quantity), 
$\gamma_t(\phi)$ is approximated by
\be
\gamma_n(\phi) = \mathbb{E}[  \phi({\bf X}_n) \prod_{p=0}^{n-1} G({\bf X}_p)]
\label{eq2}
\ee
with $G({\bf x})= e^{-\tau V({\bf x})}$ and initial distribution $\eta_0({\bf x})= \psi({\bf x},0)$.
As seen, $\gamma_n(\phi)$ is the unrenormalized Feynman-Kac measure introduced in this work, see section \ref{FKSG}, with operator $Q$ given by $e^{-tH}$. 

In physics where the number of degrees of freedom $d$ of the quantum systems studied is large (typically, $d$ is proportional to the number of 
{\it physical} particles), DMC calculations using Eq.(\ref{eq2}) are just unfeasible because of the uncontrolled fluctuations 
of $G({\bf x})$. In practice, this problem is solved by introducing importance sampling.
Let $\psi_G:\mathsf{E}\rightarrow\mathbb{R}^+$ be a so-called guiding wavefunction (in practice, a good approximation of the ground-state wavefunction).
Let us introduce the distribution $f({\bf x},t)= \psi_G({\bf x}) \psi({\bf x},t)$. 
The partial differential equation obeyed by $f$ is easily shown to be
$$
\frac{ \partial f({\bf x},t) }{\partial t} = L f({\bf x},t) -E_L({\bf x}) f({\bf x},t)
$$
where $L$ is a Fokker-Planck operator written as
\be
L = \frac{1}{2} \nabla^2 - \nabla [ {\bf b} .]
\label{FP}
\ee
where the drift vector is ${\bf b}({\bf x})=\frac{ \nabla \psi_G}{\psi_G}$ and $E_L({\bf x})$ is a new potential function, called the local energy, given by
\be
E_L({\bf x}) = \frac{ H \psi_G}{\psi_G}.
\label{el}
\ee
The FK formula, Eq.(\ref{eq2}), becomes
$$
\gamma_n(\psi_G \phi) = \mathbb{E}[  \phi({\bf X}_n) \prod_{p=0}^{n-1} G({\bf X}_p)]
$$
where the expectation, $\mathbb{E}$, now refers to the law of the {\it drifted} Brownian process associated with the 
Fokker-Planck operator, Eq.(\ref{FP}), the weight is $G(x)= e^{-\tau E_L({\bf x})}$, and the initial condition is $\eta_0({\bf x})= 
\psi_G({\bf x}) \psi({\bf x},0)$. Note that the operator $Q$ is now given by $Q=e^{-\tau \psi_G H \frac{1}{\psi_G}}$.

Introducing importance sampling has two important consequences which make in practice the DMC simulations feasible.
First, the drift contribution of the Markow kernel 
"pushes" the particles toward the regions where $\Psi_G$ is important (large), thus avoiding to accumulate statistics in regions where the product
$\prod_{p=0}^{n-1} G({\bf X}_p)$ is (very) small.
Second, the statistical fluctuations of the weights are much reduced since the variations of the local energy, Eq.(\ref{el}), are directly related to 
the quality of the approximate guiding wavefunction (no fluctuations when $\psi_G$ is the exact ground-state).

\section{Proofs}

\subsection{Structure of Appendix}

This appendix is split into two sections.  The first is Section \ref{app:tech_mc} which considers
some technical results for Markov chains,  which are applied particularly for the proof of Theorem \ref{prop:main_prop} and also  Section \ref{app:clt} used for the proof of Theorem \ref{theo:clt1}.
Section \ref{app:l1_proof} houses all the proofs for Theorem \ref{theo:l1}. Throughout this appendix
$C$ is generic finite constant 
whose value may change from line-to-line of the computations 
and any dependencies in terms of the parameters of DMC or of FK models
will be made clear in each context.

\subsection{Technical Results for Markov Chains}\label{app:tech_mc}

The notations in this section,  in particular the symbols used,  should be taken as independent of the main text.
The reason for this is that these results are of independent interest and do not need the notation of DMC to be written.
We consider a time-homogenous Markov chain on general state-space $(\mathsf{X},\mathcal{X})$
of initial distribution $\mu$ and kernel $K$, with the latter having invariant measure $\pi$. Let $l\in\mathbb{N}$ be fixed and consider for $\varphi\in\mathcal{B}_b(\mathsf{X})$ 
$$
\frac{1}{n}\sum_{q=1}^n \varphi\left(X_{q:q+l-1}\right).
$$
We shall prove, under assumptions, convergence of the above quantity to
$$
\pi\otimes K^{\otimes(l-1)}(\varphi) = \int_{\mathsf{X}^l} \varphi(x_{1:l})\pi(dx_1)\prod_{j=2}^{l}K(x_{j-1},dx_j)
$$
where we use the notation $x_{1:l}=(x_1,\dots,x_l)$.

\begin{hypH}\label{ass:mc1}
There exists a $(\epsilon,\nu)\in(0,1)\times\mathcal{P}(\mathsf{X})$ such that for any $(x,A)\in\mathsf{X}\times\mathcal{X}$ we have
$$
\int_A K(x,dy) \geq \epsilon\int_A \nu(dy).
$$
\end{hypH}

\begin{prop}\label{prop:mc}
Assume (H\ref{ass:mc1}). Then for any $p\geq 1$ there exist a $C<+\infty$ such that for any $(n,l,\varphi)\in\mathbb{N}^2\times\mathcal{B}_b(\mathsf{X}^l)$:
$$
\mathbb{E}\left[\left|\frac{1}{n}\sum_{q=1}^n \varphi\left(X_{q:q+l-1}\right)-\pi\otimes K^{\otimes(l-1)}(\varphi)\right|^p\right]^{1/p} \leq \frac{C\|\varphi\|l\epsilon^{-1}}{\sqrt{n}}
$$
where $\epsilon$ is as (H\ref{ass:mc1}) and $\|\varphi\|=\sup_{x\in\mathsf{X}^l}|\varphi(x)|$.
\end{prop}

\begin{proof}
Define the function
\begin{equation}\label{eq:path_pois}
\widehat{\varphi}(x_{1:l}) = \sum_{q\geq 0}\left\{(K^{\otimes l})^q(\varphi)(x_{1:l})-\pi\otimes K^{\otimes(l-1)}(\varphi)\right\}
\end{equation}
where for $q\in\mathbb{N}$
\begin{align*}
(K^{\otimes l})^q(\varphi)(x_{1:l}) & = \int_{\mathsf{X}^{lq}} \varphi(x_{1:l}^{q}) \prod_{j=1}^q K^{\otimes l}(x_{1:l}^{j-1},dx_{1:l}^j) \\
& = \int_{\mathsf{X}^{lq}} \varphi(x_{1:l}^{q}) \prod_{j=1}^q\left\{\prod_{k=1}^l K(x_{k-1}^j,dx_k^j)\right\}
\end{align*}
with $x_{1:l}^0=x_{1:l}$ and $x_0^j=x_l^{j-1}$ (notice that there is $0$ subscript of $x$ in the first line).
In the summand in \eqref{eq:path_pois} if $q=0$ we set $(K^{\otimes l})^q(\varphi)(x_{1:l})=\varphi(x_{1:l})$.
Note that because of (H\ref{ass:mc1}) it easily follows that 
$$
\sup_{x_{1:l}\in\mathsf{X}_l}|\widehat{\varphi}(x_{1:l})| \leq 2\epsilon^{-1} \|\varphi\|.
$$
Then it is easy to check that for any $x_{1:l}\in\mathsf{X}^l$
\begin{align*}
\varphi(x_{1:l}) - \pi\otimes K^{\otimes(l-1)}(\varphi) & = \widehat{\varphi}(x_{1:l}) - K^{\otimes l}(\widehat{\varphi})(x_{1:l})\\
& = \widehat{\varphi}(x_{1:l}) - K^{\otimes l}(\widehat{\varphi})(x_{l}).
\end{align*}
Now, denoting the natural filtration of the Markov chain $(X_k)_{k\geq 0}$ as $(\mathcal{F}_k)_{k\geq 0}$, we remark that
$$
\mathbb{E}[\varphi\left(X_{q:q+l-1}\right)|\mathcal{F}_{q-1}] = K^{\otimes l}(\widehat{\varphi})(x_{q-1}).
$$
Therefore it easily follows that
\begin{align}
\frac{1}{n}\sum_{q=1}^n\left\{ \varphi\left(X_{q:q+l-1}\right)-\pi\otimes K^{\otimes(l-1)}(\varphi)\right\} & = 
\frac{1}{n}\sum_{q=1}^n\left\{\widehat{\varphi}(x_{q:q+l-1}) - K^{\otimes l}(\widehat{\varphi})(x_{q})\right\}\nonumber\\
& = \frac{1}{n}\left\{M_n + R_n\right\}\label{eq:main_eq}
\end{align}
where
\begin{eqnarray*}
M_n & = & \sum_{q=1}^n\left\{\widehat{\varphi}(x_{q:q+l-1}) - K^{\otimes l}(\widehat{\varphi})(x_{q-1})\right\} \\
R_n & = & \left\{K^{\otimes l}(\widehat{\varphi})(x_{0})-K^{\otimes l}(\widehat{\varphi})(x_{n})\right\}.
\end{eqnarray*}
Then combining \eqref{eq:main_eq} with the Minkowksi inequality we have that
\begin{equation}\label{eq:main_eq1}
\mathbb{E}\left[\left|\frac{1}{n}\sum_{q=1}^n \varphi\left(X_{q:q+l-1}\right)-\pi\otimes K^{\otimes(l-1)}(\varphi)\right|^p\right]^{1/p} \leq 
\mathbb{E}\left[\left|\tfrac{1}{n}M_n\right|^p\right]^{1/p} +
\frac{C\|\varphi\|\epsilon}{n}
\end{equation}
Now we adopt the notation for any $(q,r)\in\{1,\dots,n\}\times\{0,\dots,l-1\}$
\begin{eqnarray}
K^{\otimes(l-r-1)}(\widehat{\varphi})(x_{q:q+r}) & := &  \int_{\mathsf{X}^{l-r-1}}\widehat{\varphi}(x_{q:q+l-1})
\prod_{j=q+r+1}^{q+l-1}K(x_{j-1},dx_j)\label{eq:prod1}\\
K^{\otimes(l-r)}(\widehat{\varphi})(x_{q:q+r-1}) & := &  \int_{\mathsf{X}^{l-r}}\widehat{\varphi}(x_{q:q+l-1})
\prod_{j=q+r}^{q+l-1}K(x_{j-1},dx_j)\label{eq:prod2}
\end{eqnarray}
with the convention that $x_{q:q-1}=x_{q-1}$.  We will also write 
\begin{equation}
K^{\otimes(0)}(\widehat{\varphi})(x_{q:q+l-1}) = \widehat{\varphi}(x_{q:q+l-1}).\label{eq:prod3}
\end{equation}
Then we have that
\begin{equation}\label{eq:mart_rep}
M_n = \sum_{q=1}^n \sum_{r=0}^{l-1} \xi_{q,r}
\end{equation}
where
$$
\xi_{q,r} := K^{\otimes(l-1-r)}(\widehat{\varphi})(x_{q:q+r})-K^{\otimes(l-r)}(\widehat{\varphi})(x_{q:q+r-1}).
$$
Now set
$$
M_n(r) := \sum_{q=1}^n \xi_{q,r}
$$
and note we clearly have $M_n=\sum_{r=0}^{l-1}M_n(r)$. For $r\in\{0,\dots,l-1\}$ fixed define the filtration
$\mathcal{F}_q^r:=\sigma(X_0,\dots,X_{q+r})$, $q\geq 1$ with $\mathcal{F}_0^r:=\sigma(X_0)$. Then it is clear that $M_n(r)$ is a $\mathcal{F}_n^r-$martingale. Then applying the Minkowski inequality $l-$times followed by the Burkholder-Gundy-Davis inequality $l-$times, we obtain the upper-bound:
$$
\mathbb{E}\left[\left|\tfrac{1}{n}M_n\right|^p\right]^{1/p} \leq \frac{C\|\varphi\|l\epsilon^{-1}}{\sqrt{n}}
$$
and then combining the above upper-bound with \eqref{eq:main_eq1} allows us to conclude.
\end{proof}


\subsubsection{Central Limit Theorems}\label{app:clt}

Recall the definitions
in \eqref{eq:path_pois} and \eqref{eq:prod1}-\eqref{eq:prod3} and set for $(l,k,\varphi)\in\{2,3,\dots\}\times\{l,l+1,\dots\}\times\mathcal{B}_b(\mathsf{X}^l)$
\begin{align}
\sigma^2_{\varphi} & :=  \int_{\mathsf{X}^{l-1}} \tilde{\varphi}(x_{k-l+1:k-1}) \pi(dx_{k-l+1})\prod_{s=k-l+2}^{k-1} K(x_{s-1},dx_s) \label{eq:clt_av1}\\
\tilde{\varphi}(x_{k-l+1:k-1}) & =  \int_{\mathsf{X}}\left(\sum_{j=0}^{l-1} K^{\otimes(l-1-j)}(\widehat{\varphi})(x_{k-j:k})\right)^2 K(x_{k-1},dx_k) - 
 \left(\sum_{j=0}^{l-1} K^{\otimes(l-j)}(\widehat{\varphi})(x_{k-j:k-1})\right)^2.\label{eq:clt_av2}
\end{align}
Below we use $\xrightarrow[]{d}$ to denote convergence in distribution as $n$ increases. 

\begin{prop}\label{prop:clt_1}
Assume (H\ref{ass:mc1}). For any $(l,\varphi)\in\{2,3,\dots\}\times\mathcal{B}_b(\mathsf{X}^l)$ we have
$$
\frac{1}{\sqrt{n}}\sum_{q=1}^n (\varphi\left(X_{q:q+l-1}\right)-\pi\otimes K^{\otimes(l-1)}(\varphi))\xrightarrow[]{d}  \mathcal{N}(0,\sigma_{\varphi}^2)
$$
where $\sigma^2_{\varphi}$ is defined in \eqref{eq:clt_av1}-\eqref{eq:clt_av2}.
\end{prop}

\begin{proof}
    We follow the notation of Proposition \ref{prop:mc}. We have $|\frac{1}{\sqrt{n}}R_n|\leq C\|\varphi\|\epsilon/\sqrt{n}\rightarrow0$. For $\frac{1}{\sqrt{n}}M_n$ we have the representation
    $$
    M_n = \frac{1}{\sqrt{n}}\sum_{k=l}^n\sum_{j=0}^{l-1}\xi_{k-j,j} + \frac{1}{\sqrt{n}}\sum_{k=1}^{l-1}\sum_{j=0}^{l-k}\xi_{k,j} + \frac{1}{\sqrt{n}}\sum_{k=n-l+1}^{n}\sum_{j=n-k+1}^{l-1}\xi_{k,j}.
    $$
    The number of terms in the second and third terms on the R.H.S.~is $\mathcal{O}(l^2)$ and all the terms are bounded thus they approach zero as $n\rightarrow\infty$. For the first term, we use the Martingale array CLT (\cite[Corollary 3.1]{hall}). Notice that $\sum_{k=l}^n\sum_{j=0}^{l-1}\xi_{k-j,j}$ is a $\mathcal{F}^0_n$-martingale.
    Let $\delta>0$, we have
    \begin{eqnarray*}
         \sum_{k=l}^n \mathbb{E}\left[\left(\frac{1}{\sqrt{n}}\sum_{j=0}^{l-1}\xi_{k-j,j}\right)^2\mathbbm{1}_{\{|x|>\delta\}}\left(\frac{1}{\sqrt{n}}\sum_{j=0}^{l-1}\xi_{k-j,j}\right)\bigg|\mathcal{F}_{k-1}^0\right]
            & \leq& \frac{1}{n^{3/2}\delta}\sum_{k=l}^n\mathbb{E}\left[\left(\sum_{j=0}^{l-1}\xi_{k-j,j}\right)^3\bigg|\mathcal{F}_{k-1}^0\right]\\
            &\leq& \frac{8l^3(n-l+1)\|\varphi\|\epsilon^{-1}}{n^{3/2}\delta}\longrightarrow0.
    \end{eqnarray*}
    For the asymptotic variance, noting \eqref{eq:clt_av2},   we use the definition of $\xi_{k-j,j}$ and write
    \begin{equation*}
\sum_{k=l}^n \mathbb{E}\left[\left(\frac{1}{\sqrt{n}}\sum_{j=0}^{l-1}\xi_{k-j,j}\right)^2\bigg|\mathcal{F}_{k-1}^0\right]
            =\frac{1}{n}\sum_{k=l}^n \tilde{\varphi}(X_{k-l+1:k-1})
\rightarrow_{\mathbb{P}}\sigma_{\varphi}^2,
    \end{equation*}
    where $\rightarrow_{\mathbb{P}}$ denotes convergence in probability as $n$ grows and this follows from Proposition \ref{prop:mc}.
\end{proof}

\begin{prop}\label{prop:clt_2}
    Assume (H\ref{ass:mc1}). For any $(l,\varphi,\psi)\in\{2,3\dots\}\times\mathcal{B}_b(\mathsf{X}^l)^2$ we have
\begin{equation*}
\sqrt{n}\left(\frac{\frac{1}{n}\sum_{q=1}^n \varphi(X_{q:q+l-1})}{\frac{1}{n}\sum_{q=1}^n \psi(X_{q:q+l-1})}-\frac{\pi\otimes K^{\otimes(l-1)}(\varphi)}{\pi\otimes K^{\otimes(l-1)}(\psi)}\right) 
\xrightarrow[]{d} \mathcal{N}(0, \sigma_{\tau}^2),
\end{equation*}
with $\tau = \frac{1}{\pi\otimes K^{\otimes(l-1)}(\psi)}\varphi - \frac{\pi\otimes K^{\otimes(l-1)}(\varphi)}{(\pi\otimes K^{\otimes(l-1)}(\psi))^2}\psi$.
\end{prop}

\begin{proof}
    Using the delta method, the limit is the same as
    \begin{equation*}
        \begin{split}
                &\frac{1}{\pi\otimes K^{\otimes(l-1)}(\psi)}\frac{1}{\sqrt{n}}\left(\sum_{q=1}^n \varphi(X_{q:q+l-1})-\pi\otimes K^{\otimes(l-1)}(\varphi)\right) \\
                -& \frac{\pi\otimes K^{\otimes(l-1)}(\varphi)}{(\pi\otimes K^{\otimes(l-1)}(\psi))^2}\frac{1}{\sqrt{n}}\left(\sum_{q=1}^n \psi(X_{q:q+l-1})-\pi\otimes K^{\otimes(l-1)}(\psi)\right).
        \end{split}    
    \end{equation*}
    Similar to Proposition \ref{prop:clt_1} we use the decomposition
    \begin{equation}
        \begin{split}
            \xi_{q,r} & = \frac{1}{\pi\otimes K^{\otimes(l-1)}(\psi)}\left(K^{\otimes(l-1-r)}(\widehat{\varphi})(x_{q:q+r})-K^{\otimes(l-r)}(\widehat{\varphi})(x_{q:q+r-1})\right)\\
            -& \frac{\pi\otimes K^{\otimes(l-1)}(\varphi)}{(\pi\otimes K^{\otimes(l-1)}(\psi))^2}\left(K^{\otimes(l-1-r)}(\widehat{\psi})(x_{q:q+r})-K^{\otimes(l-r)}(\widehat{\psi})(x_{q:q+r-1})\right)\\
            =& K^{\otimes(l-1-r)}(\widehat{\tau})(x_{q:q+r})-K^{\otimes(l-r)}(\widehat{\tau})(x_{q:q+r-1})
        \end{split}
    \end{equation}
and doing similar calculations we get the desired result.
\end{proof}

\begin{prop}\label{prop:clt_3}
Assume (H\ref{ass:mc1}). Let $(\tilde{X}_q)_{q\geq 0}$ be an independent copy of $(X_q)_{q\geq 0}$. 
For any $(l,\varphi,\psi)\in\{2,3\dots\}\times\mathcal{B}_b(\mathsf{X}^l)^2$ we have
$$
\sqrt{n}\left(\frac{\frac{1}{n}\sum_{q=1}^n \varphi(X_{q:q+l-1})}{\frac{1}{n}\sum_{q=1}^n \psi(\tilde{X}_{q:q+l-1})}-\frac{\pi\otimes K^{\otimes(l-1)}(\varphi)}{\pi\otimes K^{\otimes(l-1)}(\psi)}\right)\xrightarrow[]{d}  \mathcal{N}(0,\sigma^2)
$$
with
$$
\sigma^2 = \frac{1}{(\pi\otimes K^{\otimes(l-1)}(\psi))^2}\sigma_{\varphi}^2 + \frac{(\pi\otimes K^{\otimes(l-1)}(\varphi))^2}{(\pi\otimes K^{\otimes(l-1)}(\psi))^4}\sigma_{\psi}^2,
$$
where $\sigma_{\varphi}^2$ and $\sigma_{\psi}^2$ are defined via \eqref{eq:clt_av1}.
\end{prop}
\begin{proof}
    By the delta method the distributional limit of
    $$
    \sqrt{n}\left(\frac{\frac{1}{n}\sum_{q=1}^n \varphi(X_{q:q+l-1})}{\frac{1}{n}\sum_{q=1}^n \psi(\tilde{X}_{q:q+l-1})}-\frac{\pi\otimes K^{\otimes(l-1)}(\varphi)}{\pi\otimes K^{\otimes(l-1)}(\psi)}\right)
    $$
    is the same as the distributional limit of 
    \begin{equation*}
        \begin{split}
                &\frac{1}{\pi\otimes K^{\otimes(l-1)}(\psi)}\frac{1}{\sqrt{n}}\left(\sum_{q=1}^n \varphi(X_{q:q+l-1})-\pi\otimes K^{\otimes(l-1)}(\varphi)\right) \\
                -& \frac{\pi\otimes K^{\otimes(l-1)}(\varphi)}{(\pi\otimes K^{\otimes(l-1)}(\psi))^2}\frac{1}{\sqrt{n}}\left(\sum_{q=1}^n \psi(\tilde{X}_{q:q+l-1})-\pi\otimes K^{\otimes(l-1)}(\psi)\right)
        \end{split}    
    \end{equation*}
    The result follows because of the independence of $(X_q)_{q\geq 0}$ and $(\tilde{X}_q)_{q\geq 0}$ and the CLT for each term. 
\end{proof}

\section{$\mathbb{L}_1-$Proofs}\label{app:l1_proof}

\begin{lem}\label{lem:decomp}
We have the following decomposition for any $(n,l,N,\varphi)\in\mathbb{N}_0\times\mathbb{N}^2\times\mathcal{B}_b(\mathsf{E})$ whenever the formulae exist:
$$
\mu^{l,N}_n(\varphi) = \frac{1}{n}\sum_{p=0}^{n-1}\eta^N_{p}Q^{l}(\varphi)+\frac{1}{\sqrt{N}}~
V^{l,N}_n(\varphi)
$$
where
\begin{eqnarray*}
V^{l,N}_n(\varphi)& := & \frac{1}{n}\sum_{p=0}^{n-1}\sum_{k=1}^l\frac{\gamma^N_{p+k}(1)}{\gamma^N_{p}(1)}~V^N_{p+k}(Q^{l-k}(\varphi))\\
V^N_n(\varphi)& := & \sqrt{N}\left(\eta^N_{n}(\varphi)-\Phi\left(\eta^N_{n-1}\right)(\varphi)\right).
\end{eqnarray*}
\end{lem}

\begin{proof}
We start by noting that
$$
\gamma^N_{p+l}-\gamma^N_{p}Q^{l}=\sum_{k=1}^l\left(\gamma^N_{p+k}Q^{l-k}-\gamma^N_{p+(k-1)}Q^{l-(k-1)}\right).
$$
Conversely,  we have
$$
\eta^N_{p+(k-1)}Q^{l-(k-1)}=\eta^N_{p+(k-1)}QQ^{l-k}=\eta^N_{p+(k-1)}(G)~\frac{\eta^N_{p+(k-1)}QQ^{l-k}}{\eta^N_{p+(k-1)}Q(1)}
=
\eta^N_{p+(k-1)}(G)~\Phi\left(\eta^N_{p+(k-1)}\right)Q^{l-k}.
$$
This yields the decomposition
$$
\gamma^N_{p+l}-\gamma^N_{p}Q^{l}=\sum_{k=1}^l\gamma^N_{p+k}(1)\left(\eta^N_{p+k}-\Phi\left(\eta^N_{p+(k-1)}\right)\right)Q^{l-k}
$$
from which we check that
$$
\frac{\gamma^N_{p+l}(\varphi)}{\gamma^N_{p}(1)}-\eta^N_{p}Q^{l}(\varphi)=\frac{1}{\sqrt{N}}\sum_{k=1}^l\frac{\gamma^N_{p+k}(1)}{\gamma^N_{p}(1)}~
V^N_{p+k}(Q^{l-k}(\varphi))
$$
with the centered local perturbation random fields
$$
V^N_n = \sqrt{N}\left(\eta^N_{n}-\Phi\left(\eta^N_{n-1}\right)\right).
$$
This implies that
$$
\mu^{l,N}_n(\varphi) =\frac{1}{n}\sum_{p=0}^{n-1}\eta^N_{p}Q^{l}(\varphi)+\frac{1}{\sqrt{N}}~
V^{l,N}_n(\varphi)
$$
with
$$
V^{l,N}_n(\varphi)  := \frac{1}{n}\sum_{p=0}^{n-1}\sum_{k=1}^l\frac{\gamma^N_{p+k}(1)}{\gamma^N_{p}(1)}~V^N_{p+k}(Q^{l-k}(\varphi))
$$
which completes the proof.
\end{proof}

\begin{lem}\label{tech:lem}
Assume (A\ref{ass:smc}). Then there exists a $C\in(0,\infty)$ such that for any $(n,N,l,\varphi)\in
\mathbb{N}_0\times\mathbb{N}^2\times\mathcal{B}_b(\mathsf{E})$:
$$
\mathbb{E}\left[\left|\frac{\mu^{l,N}_n(\varphi)-\mu^{l}_n(\varphi)}{\mu^{l}_n(1)}\right|\right]
\leq \frac{C\|\varphi\|l}{\sqrt{N}}.
$$
\end{lem}

\begin{proof}
Using Lemma \ref{lem:decomp} and the fact that
$$
\mu^{l}_n(\varphi)=
\frac{1}{n}\sum_{p=0}^{n-1}\eta_{p}(Q^{l}(\varphi))
$$
we have that by the Minkowski inequality
$$
\mathbb{E}\left[\left|\frac{\mu^{l,N}_n(\varphi)-\mu^{l}_n(\varphi)}{\mu^{l}_n(1)}\right|\right]
\leq T_1 + T_2
$$
where 
\begin{eqnarray*}
T_1 & = & \frac{1}{n\mu^{l}_n(1)}\sum_{p=0}^{n-1}\mathbb{E}\left[\left|
\eta_p^{N}(Q^{l}(\varphi)) - \eta_{p}(Q^{l}(\varphi))\right|\right] \\
T_2 & = &  \frac{1}{n\mu^{l}_n(1)}
\sum_{p=0}^{n-1}\sum_{k=1}^l \mathbb{E}\left[\left|
\frac{\gamma^N_{p+k}(1)}{\gamma^N_{p}(1)}\left\{
\eta^N_{p+k}(Q^{l-k}(\varphi))-\Phi\left(\eta^N_{p+k-1}\right)(Q^{l-k}(\varphi))\right\}
\right|\right].
\end{eqnarray*}
We deal with $T_1$ and $T_2$ individually to conclude.

For $T_1$ one has
$$
T_1 =  \frac{1}{n\mu^{l}_n(1)}\sum_{p=0}^{n-1}\eta_{p}(Q^{l}(1))\mathbb{E}\left[\left|
\eta_p^{N}\left(\frac{Q^{l}(\varphi)}{\eta_{p}(Q^{l}(1))}\right) - \eta_{p}\left(\frac{Q^{l}(\varphi)}{\eta_{p}(Q^{l}(1))}\right)\right|\right].
$$
By using standard $\mathbb{L}_1-$bounds for Feynman-Kac formula (e.g.~\cite[Theorem 7.4.4.]{delm})  it follows that
$$
T_1 \leq  C\frac{1}{n\sqrt{N}\mu^{l}_n(1)}\sum_{p=0}^{n-1}\eta_{p}(Q^{l}(1))\left\|\frac{Q^{l}(\varphi)}{\eta_{p}(Q^{l}(1))}\right\|.
$$
By (A\ref{ass:smc}) one can show that for any $s\in\mathbb{N}_0$
\begin{equation}\label{eq:eig_bound}
\sup_{p\geq 0}\left\|\frac{Q^{s}(\varphi)}{\eta_{p}(Q^{s}(1))}\right\| \leq C\|\varphi\|
\end{equation}
where $C$ does not depend upon $s$; see \cite[Lemma 4.1]{cerou} for example. Therefore we have shown that
$$
T_1 \leq \frac{C\|\varphi\|}{\sqrt{N}}.
$$

For $T_2$ using the conditional i.i.d.~property of the particle system one has
$$
T_2 \leq   \frac{C}{n\sqrt{N}\mu^{l}_n(1)}
\sum_{p=0}^{n-1}\sum_{k=1}^l \eta_{p}(Q^{l-k}(1))
\left\|\frac{Q^{l-k}(\varphi)}{\eta_{p}(Q^{l-k}(1))}\right\|
\mathbb{E}\left[\frac{\gamma^N_{p+k}(1)}{\gamma^N_{p}(1)}\right]
$$
and then using \eqref{eq:eig_bound}
$$
T_2 \leq   \frac{C\|\varphi\|}{n\sqrt{N}\mu^{l}_n(1)}
\sum_{p=0}^{n-1}\sum_{k=1}^l \eta_{p}(Q^{l-k}(1))\mathbb{E}\left[\frac{\gamma^N_{p+k}(1)}{\gamma^N_{p}(1)}\right].
$$
Now 
$$
\mathbb{E}\left[\frac{\gamma^N_{p+k}(1)}{\gamma^N_{p}(1)}\right] = 
\eta_{p}(Q^{k}(1))\mathbb{E}\left[\eta_p^N\left(\frac{Q^k(1)}{\eta_{p}(Q^{k}(1))}\right)\right]
$$
so combining with \eqref{eq:eig_bound} yields
$$
T_2 \leq   \frac{C}{n\sqrt{N}\mu^{l}_n(1)}
\sum_{p=0}^{n-1}\sum_{k=1}^l \eta_{p}(Q^{l-k}(1))\eta_{p}(Q^{k}(1)).
$$
Now due to (A\ref{ass:smc}) there exist finite and positive constants $C,\bar{C}$ so that for any $(x,y)\in\mathsf{E}^2$
and any  $s\in\mathbb{N}_0$
$$
\bar{C} Q^s(1)(y) \leq Q^s(1)(x) \leq C Q^s(1)(y)
$$
so that it follows that there exists a finite $C$ that does not depend on $p,k,l$ or $\eta_p$ so that
$$
\frac{\eta_{p}(Q^{l-k}(1))\eta_{p}(Q^{k}(1))}{\eta_p(Q^l(1))} \leq C.
$$
Hence
$$
T_2 \leq   \frac{C\|\varphi\|l\mu^{l}_n(1)}{\sqrt{N}\mu^{l}_n(1)} = \frac{C\|\varphi\|l}{\sqrt{N}}
$$
and this concludes the proof.
\end{proof}

\begin{rem}
Extending Lemma \ref{tech:lem} and hence Theorem \ref{theo:l1} to $\mathbb{L}_q-$is challenging as
one has to control the moment
$$
\mathbb{E}\left[\left|\frac{\gamma^N_{p+k}(1)}{\gamma^N_{p}(1)}\right|^{q}\right]^{1/q}.
$$
This is particularly simple when $q=1$, but even the case $q=2$ is not trivial; see \cite{cerou} for instance.
\end{rem}


\end{document}